\documentclass[12pt]{amsart}
\usepackage{fullpage,url,amssymb,enumerate,colonequals}

\usepackage{mathrsfs} 
\usepackage{MnSymbol}
\usepackage{extarrows}
\usepackage{lscape}
\usepackage[all,cmtip]{xy}

\usepackage[OT2,T1]{fontenc}

\usepackage{color}

\usepackage[
        colorlinks, citecolor=darkgreen,
        backref,
        pdfauthor={Alain Kraus}, 
]{hyperref}

\usepackage[french]{babel}

\newcommand{\F}{\mathbb{F}}

\newcommand{\Q}{\mathbb{Q}}
\newcommand{\R}{\mathbb{R}}
\newcommand{\Z}{\mathbb{Z}}

\newcommand{\sym}{\mathbb{S}}

\newcommand{\fq}{\mathfrak{q}}

\newcommand{\cc}{\mathfrak{c}}

\DeclareMathOperator{\Aut}{Aut}

\DeclareMathOperator{\Cl}{Cl}

\DeclareMathOperator{\Gal}{Gal}

\numberwithin{equation}{section}

\newtheorem{theorem}{Th\'eor\`eme}
\newtheorem{lemme}{Lemme}

\newtheorem{proposition}{Proposition}

\theoremstyle{definition}

\theoremstyle{remark}
\newtheorem{remark}[equation]{Remarque}

\definecolor{darkgreen}{rgb}{0,0.5,0}

\begin{document}

\title{Courbes de Fermat et principe de Hasse}

\author{Alain Kraus}

\begin{abstract}   Soit $p\geq 3$ un nombre premier. 
Une courbe de Fermat  sur $\Q$ d'exposant $p$ est d\'efinie par une \'equation de la forme $ax^p+by^p+cz^p=0$ o\`u $a,b,c$ sont
des nombres  rationnels non nuls. On  d\'emontre dans cet article  qu'il  existe  une infinit\'e de courbes de Fermat d\'efinies sur $\Q$,
d'exposant $p$, deux \`a deux non $\Q$-isomorphes, contredisant le principe de Hasse.
\end{abstract}



\date{\today}

\address{Sorbonne Universit\'e,
Institut de Math\'ematiques de Jussieu - Paris Rive Gauche,
UMR 7586 CNRS - Paris Diderot,
4 Place Jussieu, 75005 Paris, 
France}
\email{alain.kraus@imj-prg.fr}

\maketitle

\section{Introduction}
Une courbe de Fermat   $C/\Q$, d'exposant  un nombre premier $p\geq 3$,  est d\'efinie par une \'equation de la forme
\begin{equation}
\label{(1.1)} 
ax^p+by^p+cz^p=0
\end{equation}
o\`u $a,b,c$ sont des nombres  rationnels non nuls.  Rappelons que, par d\'efinition,  $C$ a une obstruction locale en un nombre premier $\ell$, si l'ensemble $C(\Q_{\ell})$ est vide. 
On dit que 
$C$ 
contredit le principe de Hasse  si $C$  n'a pas d'obstructions locales  et si $C(\Q)$ est vide. 
Les places \`a l'infini n'interviennent pas ici car $C(\R)$ n'est pas vide.

Historiquement, le premier exemple de courbe de Fermat contredisant le principe de Hasse a \'et\'e d\'ecouvert   par Selmer en 1951 (\cite{Selmer}), avec  la courbe d'\'equation
$$3x^3+4y^3+5z^3=0.$$

 L'objectif de cet article est de d\'emontrer le r\'esultat suivant.
 
 \begin{theorem} \label{T:thm1}  Soit $p\geq 3$ un nombre premier. Il existe une infinit\'e de courbes de Fermat  d\'efinies sur $\Q$, d'exposant $p$, deux \`a deux non $\Q$-isomorphes,  
contredisant le principe de~Hasse.
\end{theorem}

Cet \'enonc\'e est une cons\'equence de la conjecture $(abc)$ si l'on a $p\geq 5$  (\cite{HK}, prop. 5.1). 
 
Une approche de ce r\'esultat   a \'et\'e  effectu\'ee  dans \cite{Kraus2016}. En particulier,  il a \'et\'e  \'etabli   si l'on a  $3\leq p\leq 7$.   Il a aussi \'et\'e d\'emontr\'e   pour  $p\in \lbrace 11,13,17,19\rbrace$, 
en admettant   l'hypoth\`ese de Riemann g\'en\'eralis\'ee GRH, j'avais omis de pr\'eciser ce point. En fait, 
la d\'emonstration du corollaire dans   le paragraphe I de \cite{Kraus2016}
utilise des calculs de nombres de classes qui sont valides a priori sous GRH  (voir \cite{Pari}).  
On peut  s'affranchir de cette hypoth\`ese avec  le th\'eor\`eme~\ref{T:thm2} obtenu ici, tout au moins si $p\neq 11$. 
Il en est de m\^eme si $p=11$  en rempla\c cant   dans l'\'enonc\'e de ce corollaire le couple $(11,373)$  par $(11,461)$.

\`A ma connaissance,  ce sont  les seuls r\'esultats qui \'etaient d\'ej\`a connus  concernant le th\'eor\`eme~\ref{T:thm1}. 
Signalons que l'on avait explicit\'e  dans \cite{HK}, en utilisant des m\'ethodes modulaires,  des exemples de courbes de Fermat  contredisant le principe de Hasse pour tout exposant $p<100$.

Pour tout nombre premier $p\geq 5$, on  d\'emontre ici
qu'il  existe un nombre premier $r$ (en fait une infinit\'e), 
 tel que la condition suivante soit satisfaite: il existe une infinit\'e de nombres premiers $q$ tels que
la courbe $C_q/\Q$ d'\'equation $x^p+qy^p+rz^p=0$ soit un contre-exemple au principe de Hasse (th\'eor\`eme~\ref{T:thm2}, lemme~\ref{L:lemme2} et proposition~\ref{P:prop1}). De plus, si $q$ et $q'$ sont des nombres premiers distincts, les courbes $C_q$  et $C_{q'}$ ne sont pas $\Q$-isomorphes (proposition~\ref{P:prop2} de l'Appendice). Cela \'etablit  alors  le th\'eor\`eme~\ref{T:thm1}.

Toute la suite  
est consacr\'ee \`a sa d\'emonstration. Elle utilise de fa\c con essentielle une version du  th\'eor\`eme figurant dans le paragraphe I de \cite{Kraus2016}.

D\'esormais, la lettre $p$ d\'esigne un nombre premier $\geq 5$  fix\'e.
\medskip

{\bf{Remerciements.}} Je remercie Nicolas Billerey et Nuno Freitas pour les remarques dont ils m'ont fait part pendant la r\'edaction de ce travail.

\section{Nombres premiers exceptionnels}
Selon la terminologie adopt\'ee dans \cite{HK} p. 195, un nombre premier $\ell$ est dit exceptionnel pour $p$ si les deux conditions suivantes sont remplies.

\begin{itemize}

\item[1)] On a $\ell\neq p$. 
\smallskip
\item[2)] Il existe $u,v,w\in \F_{\ell}^*$ tels que la courbe d'\'equation $ux^p+vy^p+wz^p=0$ n'ait pas de points rationnels sur $\F_{\ell}$. 
\end{itemize}

 Notons $S(p)$
l'ensemble des nombres premiers exceptionnels pour $p$. 

Le genre d'une courbe de Fermat d'exposant $p$ est $(p-1)(p-2)/2$. D'apr\`es les travaux de Weil (\cite{Weil}, p. 70, cor. 3),  si $\ell\in S(p)$  on a donc 
$$\ell<\left((p-1)(p-2)\right)^2.$$
En particulier, $S(p)$ est un ensemble  fini.

Rappelons le r\'esultat bien connu suivant. 
Soient $a,b,c$ des entiers rationnels non nuls et $C/\Q$ la courbe de Fermat d'\'equation~\eqref{(1.1)}.
\smallskip
\begin{lemme} \label{L:lemme1} Soit $\ell$ un nombre premier. Supposons que $\ell$ ne divise pas $pabc$ et  que $\ell$  n'appartienne pas \`a  $S(p)$. Alors,   $C(\Q_{\ell})$ est non vide. 
\end{lemme} 

\begin{proof}  
Notons $v_{\ell}$ la valuation standard de $\Q_{\ell}$.  D'apr\`es les   hypoth\`eses faites sur $\ell$, l'ensemble  $C(\F_{\ell})$ n'est pas vide.  
Ainsi, il existe  $x,y,z$ dans $\Z_{\ell}$ tels que l'on ait 
$$ax^p+by^p+cz^p\equiv 0 \pmod{\ell} \quad \text{et}\quad \min\left(v_{\ell}(x),v_{\ell}(y),v_{\ell}(z)\right)=0.$$
Supposons par exemple $v_{\ell}(x)=0$.  Consid\'erons le polyn\^ome $F=aX^p+by^p+cz^p\in \Z_{\ell}[X]$. On a $v_{\ell}(F(x))\geq 1$. Parce que $\ell$ ne divise  pas $pa$, on a  $v_{\ell}(F'(x))=0$. D'apr\`es le lemme
de Hensel, $F$ a donc une racine $\alpha\in \Z_{\ell}$. On a $[\alpha,y,z]\in C(\Q_{\ell})$, donc $C(\Q_{\ell})$ n'est pas vide, d'o\`u le lemme. 
\end{proof} 

\section{Densit\'e et contre-exemples au principe de Hasse \label{3}}
Soit $r$ un nombre premier satisfaisant les deux conditions suivantes:

\begin{itemize}

\item[1)]   On a $r^{p-1}\not\equiv 1 \pmod {p^2}$. 
\smallskip
\item[2)] On a $r\equiv -1 \pmod p$. 
\end{itemize}
Il existe une infinit\'e de tels nombres premiers $r$ (lemme~\ref{L:lemme2}).

 Pour tout nombre premier $q$,   notons $C_q/\Q$ la courbe de Fermat d'\'equation 
\begin{equation}
\label{(3.1)} 
x^p+q y^p+rz^p=0.
\end{equation} 
Posons 
$$K=\Q\left(\root{p}\of {r}\right).$$
Notons $O_K$ l'anneau d'entiers de $K$, $\Cl(K)$ son  groupe des classes  et $e$ l'exposant du groupe ab\'elien $\Cl(K)$. 

Par hypoth\`ese, on a $r\equiv -1 \pmod p$ et  $p\geq 5$. D'apr\`es les  r\'esultats de Limura  (\cite{Limura},  cor.  p. 158), il en r\'esulte que l'on a 
\begin{equation}
 \label{(3.2)} 
e\equiv 0 \pmod p.
\end{equation} 
(Cela  n'est pas valable pour  $p=3$, par exemple le nombre de classes de $\Q(\root{3}\of {5})$ vaut $1$.)

Soit $S$ l'ensemble des nombres premiers $q\not\equiv 1 \pmod p$  tels que la courbe $C_q/\Q$ n'ait pas  d'obstructions locales.  
D'apr\`es le th\'eor\`eme de la progression arithm\'etique de Dirichlet,  
$S$ a une densit\'e que l'on notera $d(S)$. 

Pour tout nombre premier $q\not\equiv 1 \pmod p$,  ne divisant pas l'indice $f$ de $\Z[\root{p}\of{r}]$ dans $O_K$, il existe un unique id\'eal premier de $O_K$ au-dessus de $q$ de degr\'e r\'esiduel $1$ (\cite{Kraus2016}, lemme 1).

Soit $S_0$ le sous-ensemble de $S$ form\'e des nombres premiers $q$, ne divisant pas $f$, tels que la condition suivante soit satisfaite:
si $\fq$ est l'id\'eal premier de $O_K$ au-dessus de $q$ de degr\'e r\'esiduel $1$,  l'id\'eal $\fq^{e/p}$ n'est pas principal. 
\medskip

\begin{theorem} \label{T:thm2}  Pour tout nombre premier $q\in S_0$ la courbe $C_q/\Q$ est un contre-exemple au principe de Hasse. Supposons que l'on ait  
 \begin{equation}
 \label{(3.3)} 
 d(S)>\frac{1}{p}.
\end{equation}
 Alors,  $S_0$ est infini. De plus, les courbes $C_q$ sont deux \`a deux non $\Q$-isomorphes. 
 \end{theorem} 

\begin{proof} C'est une cons\'equence directe du th\'eor\`eme \'enonc\'e dans le paragraphe I de \cite{Kraus2016}.  En effet, on a $r^{p-1}\not\equiv 1 \pmod{p^2}$ (condition 1) et $p$ divise l'ordre de $\Cl(K)$  (condition~\eqref{(3.2)}). Par ailleurs, d'apr\`es  l'in\'egalit\'e~\eqref{(3.3)}, 
l'hypoth\`ese relative \`a la  densit\'e intervenant  dans  le   th\'eor\`eme de {\it{loc. cit.}} est satisfaite. Le fait que les courbes $C_q$ soient deux \`a deux non $\Q$-isomorphes r\'esulte de  la proposition~\ref{P:prop2} de l'Appendice, d'o\`u le r\'esultat.
\end{proof} 

\begin{remark}  La premi\`ere assertion du  th\'eor\`eme~\ref{T:thm2} permet  pour des petits nombres premiers $p$, disons $p\leq 23$,  sous l'hypoth\`ese de Riemann g\'en\'eralis\'ee GRH, 
d'expliciter des  courbes de Fermat contredisant le principe de Hasse. Pour cela,  il  faut  conna\^itre $e$ et \^etre en mesure de d\'eterminer des nombres premiers de $S_0$, ce qui  n\'ecessite  souvent d'admettre  GRH. Pour  $p=5$ ou  $p=7$,  et si $r$ n'est trop grand,  il est assez simple   d'expliciter   inconditionnellement de telles courbes de Fermat.
\`A titre indicatif, les courbes  d'\'equations
$$x^5+7y^5+29z^5=0\quad \text{et}\quad x^7+19y^7+13z^7=0$$
 sont des contre-exemples au principe de Hasse.  Ils  sont analogues \`a ceux  figurant dans le corollaire 6.4.11 de \cite{Cohen}. On a la m\^eme conclusion sous GRH, avec celle
d'\'equation
$$x^{19}+7y^{19}+37z^{19}=0.$$
Cela \'etant, avec  des arguments modulaires, on peut   conclure pour cette courbe sans GRH. 
\end{remark} 

\begin{remark} La condition~\eqref{(3.2)} 
 est encore satisfaite  avec  un nombre premier $r$ congru \`a $1$ modulo $p$  (\cite{Limura},  cor.  p. 158).  Par ailleurs, si $q$ est un nombre premier distinct de $r$ et $p$, 
 la courbe $C_q$ d'\'equation~\eqref{(3.1)} n'a pas d'obstruction locale en $r$ si et seulement si la classe de $q$ est une puissance $p$-i\`eme dans $\F_r^*$ (voir la d\'emonstration du lemme~\ref{L:lemme3}).  Cette condition est r\'ealis\'ee si $r\equiv -1 \pmod p$, mais pas n\'ecessairement si $r\equiv 1 \pmod p$. Afin d'obtenir une densit\'e $d(S)$ assez grande pour que l'on ait l'in\'egalit\'e~\eqref{(3.3)}, 
cela justifie le fait que l'on ait impos\'e la condition  $r\equiv -1 \pmod p$. 
\end{remark} 
\section{choix du nombre premier $r$ \label{4}}
Compte tenu du th\'eor\`eme~\ref{T:thm2},  il s'agit de d\'emontrer qu'il existe  un nombre premier $r$  v\'erifiant les conditions 1 et 2  du paragraphe~\ref{3}, pour lequel l'in\'egalit\'e~\eqref{(3.3)}  soit r\'ealis\'ee.
\subsection{L'entier $N$} Posons  
\begin{equation}
\label{(4.1)}
N=p\prod_{\ell \in S(p)}\ell.
\end{equation}
\begin{lemme} \label{L:lemme2}   Il existe une infinit\'e de nombres premiers $r$ tels que l'on ait $r\equiv -1 \pmod N$ et $r^{p-1}\not\equiv 1 \pmod{p^2}$. 
\end{lemme} 
\begin{proof} 
Le groupe $(\Z/p^2\Z)^*$ \'etant cyclique d'ordre $p(p-1)$, il poss\`ede exactement  $p-1$ \'el\'ements  $x$  tels que $x^{p-1}=1$.  Par ailleurs, il  y a $p$ \'el\'ements dans $(\Z/p^2\Z)^*$  dont la r\'eduction modulo $p$ vaut $-1$, \`a savoir les classes d'entiers  de la forme $-1+\lambda p$ avec $0\leq \lambda\leq p-1$.  Il y a donc une telle classe $\cc$  telle que $\cc^{p-1}\neq 1$.  Ainsi, il  existe   $a\in \Z$ tel que  l'on ait
$$a\equiv -1 \pmod p\quad \text{et}\quad a^{p-1}\not\equiv 1 \pmod {p^2}.$$ 
Par d\'efinition de l'ensemble $S(p)$, 
l'entier $N/p$ est premier \`a $p$.  Il  existe donc $c\in \Z$ tel que  
$$c\equiv -1\pmod {N/p}\quad \text{et} \quad c\equiv a \pmod {p^2}.$$
 On a  en particulier $c\equiv -1 \pmod p$. Par suite,  on a 
 $$c\equiv -1 \pmod N\quad \text{et}\quad c^{p-1}\not\equiv 1 \pmod{p^2}.$$
L'entier $c$ est premier avec $Np$ (qui est multiple de $p^2$). D'apr\`es le th\'eor\`eme de Dirichlet,  il existe donc une infinit\'e de nombres premiers $r$ tels que l'on ait $$r\equiv c \pmod{Np}.$$ Pour un tel nombre premier $r$, on a  ainsi
$$r\equiv c \pmod N\quad \text{et}\quad r\equiv c\pmod {p^2}.$$
On obtient  les conditions $r\equiv -1\pmod N$ et $r^{p-1}\not\equiv 1 \pmod{p^2}$, d'o\`u  le lemme.
\end{proof} 
\subsection{Choix de $r$} Fixons d\'esormais un nombre premier  $r$ tel que l'on ait 
\begin{equation}
\label{(4.2)}
r\equiv -1 \pmod N\quad \text{et}\quad r^{p-1}\not\equiv 1 \pmod{p^2}.
\end{equation}

Cela est possible d'apr\`es le lemme~\ref{L:lemme2}. La suite de la d\'emonstration du th\'eor\`eme~\ref{T:thm1} consiste \`a \'etablir qu'avec ce choix de $r$, on obtient  l'in\'egalit\'e \eqref{(3.3)}. 
\medskip
\begin{lemme} \label{L:lemme3}
Soit $q$ un nombre premier, distinct de $p$,  tel que $q\not\equiv 1 \pmod p$.  La courbe  $C_q/\Q$ d'\'equation~\eqref{(3.1)} n'a pas d'obstructions locales  en dehors de $p$. 
\end{lemme} 

\begin{proof} 
Soit $\ell\neq p$ un nombre premier. Il s'agit de montrer que $C_q(\Q_{\ell})$ n'est pas vide. D'apr\`es le lemme~\ref{L:lemme1}, tel est le cas si $\ell$ n'est pas dans $S(p)\cup \lbrace q,r\rbrace$.

Si $q=r$ le point $[0,1,-1]$ est dans $C_q(\Q)$. On peut donc supposer $q\neq r$. 

Supposons $\ell\in S(p)$. Alors, $\ell$ divise l'entier $N$ d\'efini par l'\'egalit\'e~\eqref{(4.1)}. D'apr\`es la condition~\eqref{(4.2)}, on a $r\equiv -1 \pmod \ell$ i.e. $r\equiv (-1)^p \pmod{\ell}$. Ainsi, $r$ est une puissance $p$-i\`eme dans $\F_{\ell}^*$. On a $\ell\neq p$, donc d'apr\`es le  lemme de Hensel   $r$ est  une puissance $p$-i\`eme dans $\Q_{\ell}$. En posant $r=t^p$ avec $t\in \Q_{\ell}$,  le point $[-t,0,1]$ est dans $C_q(\Q_{\ell})$, d'o\`u l'assertion dans ce cas.

Supposons $\ell=q$. Par hypoth\`ese,  on a $q\not\equiv 1 \pmod p$. Le morphisme $\F_q^*\to \F_q^*$ qui \`a $x$ associe $x^p$  est donc bijectif. On a $q\neq r$ donc $r$ est une puissance $p$-i\`eme dans $\F_q^*$. Parce que l'on a  $q\neq p$, le lemme de Hensel entra\^ine  alors que    $r$ une puissance $p$-i\`eme dans $\Q_{q}$ et donc que $C_q(\Q_q)$ est    non vide.

Supposons  $\ell=r$. 
On a $r\equiv -1 \pmod N$ d'o\`u  $r\not\equiv 1 \pmod p$. On a $q\neq r$ et $p\neq r$. Comme ci-dessus,   on en d\'eduit que $q$ est une puissance $p$-i\`eme dans $\Q_r$, puis que $C_q(\Q_r)$ est non vide, d'o\`u le lemme.
\end{proof}  
\section{Minoration de $d(S)$ - Fin de la d\'emonstration \label{5}}
Le nombre premier $r$ v\'erifiant la condition~\eqref{(4.2)}, d\'emontrons l'\'enonc\'e suivant.

\begin{proposition} \label{P:prop1} On a l'in\'egalit\'e
\begin{equation}
\label{(5.1)} 
d(S)\geq \frac{2(p-2)}{p(p-1)}.
\end{equation}
\end{proposition} 

\begin{proof} D'apr\`es le lemme~\ref{L:lemme3}, pour tout nombre premier  $q\not\equiv 1 \pmod {p}$, la seule obstruction locale possible de la courbe $C_q$ est en $p$. Par suite, 
$d(S)$ est la densit\'e des nombres premiers $q\not\equiv 1 \pmod p$ tels que $C_q(\Q_p)$ soit non vide.

Soit $A$ l'ensemble des nombres premiers $q\not\equiv 1\pmod p$ tels que $q$ ou  $q/r$ soit  une puissance $p$-i\`eme dans $\Q_p$. Pour tout $q\in A$, l'ensemble $C_q(\Q_p)$  est non vide.

Soit  $B$ l'ensemble des nombres premiers $q\not\equiv 1\pmod p$  tels que
$q$ ou   $q/r$ soit  une puissance $p$-i\`eme  modulo $p^2$. Une unit\'e de $\Z_p$ est une puissance $p$-i\`eme dans $\Q_p$ si et seulement si  c'est une puissance $p$-i\`eme modulo $p^2$
(cf. \cite{Serre1},  p. 219, prop. 9). Il en r\'esulte que l'on a 
$$A=B.$$ 
Notons $d(B)$ la densit\'e de $B$. On d\'eduit   ce qui pr\'ec\`ede que l'on a 
$$d(S)\geq d(B).$$
D\'emontrons alors que l'on~a 
\begin{equation}
\label{(5.2)} 
d(B)=\frac{2(p-2)}{p(p-1)},
\end{equation}
ce qui  \'etablira  le r\'esultat.



Le sous-groupe des puissances $p$-i\`emes de  $(\Z/p^2\Z)^*$ est form\'e des classes  $x$  tels que $x^{p-1}=1$. Il  est d'ordre $p-1$. 

Soit $\cc\in (\Z/p^2\Z)^*$  tel que $\cc$ soit  une puissance $p$-i\`eme et que la r\'eduction de $\cc$ modulo $p$ soit  la classe de $1$ dans $(\Z/p\Z)^*$.
V\'erifions que $\cc$ est la classe de $1$ modulo $p^2$. Il existe  $\lambda\in \Z$ tel que  $1+\lambda p$ soit un  repr\'esentant de $\cc$. On a $\cc^{p-1}=1$ i.e. on a la congruence
$$(1+\lambda p)^p\equiv 1+\lambda p\pmod{p^2}.$$ 
Or on  a $(1+\lambda p)^p\equiv 1 \pmod {p^2}$,   d'o\`u  $\lambda\equiv 0 \pmod p$ et l'assertion. 
Ainsi, il  y a exactement $p-2$ classes  dans $(\Z/p^2\Z)^*$ qui sont des puissances $p$-i\`emes et dont la r\'eduction modulo $p$ ne soit  pas  la classe de $1$. 
 
Notons  $\overline r$ la classe de $r$ modulo $p^2$. Il y a $p-1$ classes $\cc$ dans $(\Z/p^2\Z)^*$ telles que $\cc/\overline r$ soit une puissance $p$-i\`eme.  V\'erifions comme ci-dessus 
qu'il existe  une unique classe $\cc\in (\Z/p^2\Z)^*$  telle que  
 $(\cc/\overline r)^{p-1}=1$ et   $\cc \pmod p=1$. 
Il existe $\mu\in \Z$  tel que $r^{p-1}\equiv 1 +\mu p\pmod{p^2}$. 
Soient $\cc$ une telle classe  et  $\lambda\in \Z$ tels que $1+\lambda p$  soit un repr\'esentant de $\cc$. On a alors 
$$(1+\lambda p)^p\equiv (1+\lambda p)r^{p-1}\pmod {p^2}.$$
On obtient $\lambda\equiv -\mu \pmod p$ i.e. $\cc$ est la classe de $1-\mu p$ modulo $p^2$. Inversement, on constate que l'on a $(1-\mu p)^{p-1}\equiv r^{p-1}\pmod {p^2}$, d'o\`u l'assertion. Il y a donc  $p-2$ classes $\cc\in (\Z/p^2\Z)^*$ telles que $\cc/\overline r$ soit une puissance $p$-i\`eme et que $\cc \pmod p\neq 1$. 
  
Par ailleurs, on a $r^{p-1}\not\equiv 1 \pmod {p^2}$, donc $\overline r$ n'est pas une puissance $p$-i\`eme dans $(\Z/p^2\Z)^*$.  Il en r\'esulte qu'il y a   $2(p-2)$ classes $\cc$ dans $(\Z/p^2\Z)^*$, distinctes deux \`a deux,  telles que $\cc$ ou $\cc/\overline r$ soit une puissance $p$-i\`eme et que $\cc \pmod p\neq 1$. D'apr\`es le th\'eor\`eme de Dirichlet, $\varphi$ \'etant la fonction indicatrice d'Euler, il y a une densit\'e  $1/\varphi(p^2)$ de nombres premiers dans chacune de ces classes. L'ensemble $B$ \'etant form\'e  des nombres premiers appartenant  \`a la r\'eunion de  ces $2(p-2)$ classes de $(\Z/p^2\Z)^*$, on en d\'eduit   l'\'egalit\'e~\eqref{(5.2)},  d'o\`u  la proposition.
\end{proof} 

\subsection{Fin de la d\'emonstration}  Le cas des courbes de Fermat d'exposant $3$  est  trait\'e  dans  \cite{Kraus2016}.
Pour $p\geq 5$, d'apr\`es la proposition~\ref{P:prop1}, on a en particulier
$$d(S)>\frac{1}{p}.$$
Le th\'eor\`eme~\ref{T:thm2}  entra\^ine alors le r\'esultat, d'o\`u le th\'eor\`eme~\ref{T:thm1}.

\section{Appendice}
Son objectif est d'\'etablir  l'\'enonc\'e suivant,  utilis\'e dans la d\'emonstration  du th\'eor\`eme~\ref{T:thm2}. 

\begin{proposition}  \label{P:prop2}  Soient $q$ et $q'$ des nombres premiers distincts.  Les courbes $C_q$ et $C_{q'}$ ne sont pas  $\Q$-isomorphes.
\end{proposition} 

On va d\'emontrer un r\'esultat  g\'en\'eral qui entra\^ine  directement la proposition~\ref{P:prop2}.

Soient $p\geq 5$ un nombre premier et $K$ un corps   de caract\'eristique $0$. Soient $b,c$ et $b',c'$ des \'elements non nuls de $K$. Notons respectivement  $C$ et $C'$ les courbes de Fermat d'\'equations 
$$x^p+by^p+cz^p=0\quad \text{et}\quad x^p+b'y^p+c'z^p=0.$$
On obtient   une condition n\'ecessaire et suffisante pour que $C$ et $C'$ soient $K$-isomorphes. Je  n'ai pas trouv\' e de r\'ef\'erence \`a ce sujet dans la litt\'erature.
\medskip
\begin{theorem}  \label{T:thm3} Les courbes $C$ et $C'$ sont $K$-isomorphes si et seulement si l'un des couples  suivants  appartient \`a $K^{*p}\times K^{*p}$:
$$\left(\frac{b}{b'},\frac{c}{c'}\right),\quad \left(\frac{b}{c'},\frac{b'}{c}\right),\quad  \left(cb',\frac{b}{cc'}\right),\quad \left(bc', \frac{b'}{cc'}\right),\quad  \left(\frac{bc'}{c}, \frac{cb'}{c'}\right),  \quad \left(cc',\frac{cb'}{b}\right).$$
\end{theorem}

\begin{proof}
On v\'erifie    directement que si l'un de ces couples est dans $K^{*p}\times K^{*p}$, alors $C$ et $C'$ sont isomorphes sur $K$.

Inversement, supposons que $C$ et $C'$ soient $K$-isomorphes. Notons $F$ la courbe de Fermat  d'\'equation
$$x^p+y^p+z^p=0.$$
 Soient $\beta$ et $\gamma$ des racines $p$-i\`emes 
de $b$ et $c$ dans une cl\^oture alg\'ebrique $\overline K$ de $K$. Soient $\beta'$ et $\gamma'$ les analogues de $\beta$ et $\gamma$ pour $b'$ et $c'$. Les courbes $C$ et $C'$ sont isomorphes \`a $F$ sur $\overline K$ via les morphismes
$\varphi : C\to F$ et $\varphi' : C'\to F$ d\'efinis par
$$\varphi([x,y,z])=[x,\beta y,\gamma z]\quad \text{et}\quad \varphi'([x,y,z])=[x,\beta' y,\gamma' z].$$

Soient $\mu_p$ le groupe des racines $p$-i\`emes de l'unit\'e et $\zeta$ un g\'en\'erateur de $\mu_p$. Soit $\Aut(F)$ le groupe des automorphismes de $F$. Il est d'ordre $6p^2$. Plus pr\'ecis\'ement,  $\Aut(F)$ est  engendr\'e par l'ensemble des morphismes
$\alpha_{j,j'}$ tels que 
$$\alpha_{j,j'}([x,y,z])=[\zeta^j x,\zeta^{j'}y,z]\quad \text{pour}\quad 0\leq j,j'\leq p-1,$$
qui est  un sous-groupe  distingu\'e isomorphe  \`a $\mu_p\times \mu_p$,  et par les automorphismes de permutation qui forment  un groupe isomorphe \`a  $\sym_3$ (\cite{TZ}, p. 173, Theorem).

Posons $G_K=\Gal(\overline K/K)$. Les classes de $K$-isomorphisme de $C$ et $C'$  correspondent \`a deux classes de cocycle de l'ensemble de cohomologie $H^1(G_K,\Aut(F))$ (\cite{Silverman}, p. 318, Theorem 2.2).
Soient 
$$\Psi, \Psi' : G_K\to \Aut(F)$$
les applications d\'efinies pour tout $\sigma\in G_K$ par les \'egalit\'es
$$\Psi(\sigma)={^{\sigma}}\varphi \varphi^{-1}\quad \text{et}\quad \Psi'(\sigma)={^{\sigma}}\varphi' \varphi'^{-1}.$$
Ce sont deux $1$-cocycles  dont les classes correspondent respectivement \`a celles de $C$ et $C'$ (cf. {\it{loc. cit.}}).
Pour tous $\sigma\in G_K$ et $[x,y,z]\in F$, on a 

\begin{equation}
 \label{(6.1)} 
\Psi(\sigma)([x,y,z])=\left[x, \frac{\sigma(\beta)}{\beta} y, \frac{\sigma(\gamma)}{\gamma} z\right]\quad \text{et}\quad \Psi'(\sigma)([x,y,z])=\left[x, \frac{\sigma(\beta')}{\beta'} y, \frac{\sigma(\gamma')}{\gamma'} z\right].
\end{equation}

Par hypoth\`ese, $\Psi$ et $\Psi'$ sont cohomologues. Il  existe  donc $\alpha\in \Aut(F)$ tel que l'on ait ({\it{loc. cit.}}, p. 421)
\begin{equation}
 \label{(6.2)} 
\Psi(\sigma) \alpha={^{\sigma}}\alpha \Psi'(\sigma).
\end{equation}

Il existe des indices $j$ et $j'$ avec $0\leq j,j'\leq p-1$,  tels que l'on ait 
$$\alpha\in \big\lbrace\alpha_{j,j'}, \alpha_{j,j'} s,\alpha_{j,j'}t\big\rbrace,$$
o\`u $s$ est un $3$-cycle et o\`u $t$ est une transposition.
On est donc amen\'e \`a consid\'erer plusieurs cas. \'Etablissons en d\'etail le r\'esultat  pour l'un d'entre eux, les autres se traitent de la m\^eme fa\c con. Notons $\chi : G_K\to \Aut(\mu_p)=(\Z/p\Z)^*$ le caract\`ere cyclotomique.
\smallskip

Supposons que l'on ait $\alpha=\alpha_{j,j'}$. D'apr\`es les \'egalit\'es~\eqref{(6.1)}, pour tout $\sigma\in G_K$, on a 

$$\Psi(\sigma) \alpha([x,y,z])=\left[ \zeta^j x,\frac{\sigma(\beta)}{\beta}\zeta^{j'}y,\frac{\sigma(\gamma)}{\gamma} z\right]=\left[x,\frac{\sigma(\beta)}{\beta}\zeta^{j'-j}y,\frac{\sigma(\gamma)}{\gamma} \zeta^{-j}z\right],$$

$${^{\sigma}}\alpha \Psi'(\sigma)([x,y,z])=\left[\zeta^{j\chi(\sigma)} x, \frac{\sigma(\beta')}{\beta'} \zeta^{j'\chi(\sigma)}y, \frac{\sigma(\gamma')}{\gamma'} z\right]=\left[x, \frac{\sigma(\beta')}{\beta'} \zeta^{(j'-j)\chi(\sigma)}y, \zeta^{-j\chi(\sigma)}\frac{\sigma(\gamma')}{\gamma'} z\right].$$

Il  r\'esulte alors de \eqref{(6.2)} que l'on a 
$$\frac{\sigma(\beta)}{\beta}\zeta^{j'-j}=\frac{\sigma(\beta')}{\beta'} \zeta^{(j'-j)\chi(\sigma)}    \quad \text{et}\quad \frac{\sigma(\gamma)}{\gamma} \zeta^{-j}= \zeta^{-j\chi(\sigma)}\frac{\sigma(\gamma')}{\gamma'}.$$
On d\'eduit de la premi\`ere \'egalit\'e que l'on a
$$ \sigma\left(\frac{\beta}{\beta'} \zeta^{j-j'}\right)=\frac{\beta}{\beta'} \zeta^{j-j'}.$$
Par suite, on a $\frac{\beta}{\beta'} \zeta^{j-j'}\in K^*$, donc $b/b'\in K^{*p}$. De m\^eme, d'apr\`es la seconde \'egalit\'e, on  a
$$ \sigma\left(\frac{\gamma}{\gamma'}\zeta^j\right)=\frac{\gamma}{\gamma'} \zeta^{j},$$
d'o\`u $\frac{\gamma}{\gamma'}\zeta^j\in K^*$  puis
$c/c'\in K^{*p}$. Ainsi, le couple $(b/b',c/c')$ appartient \`a $K^{*p}\times K^{*p}$ comme annonc\'e. 

Soit $s$  le $3$-cycle d\'efini par $$s([x,y,z]=[z,x,y].$$ Si $\alpha=\alpha_{j,j'} s$  ou  $\alpha=\alpha_{j,j'} s^2$, on v\'erifie    que
$(bc',b'/cc')$  ou $(cb',b/cc')$ est dans $K^{*p}\times K^{*p}$.

Si  $t$ est l'une des transpositions d\'efinies par  
$$t([x,y,z])=[y,x,z],\quad  [x,z,y]\quad \text{ou}\quad  [z,y,x],$$  
on constate respectivement que 
$$(bc'/c,cb'/c'),\quad (b/c',b'/c)\quad \text{ou}\quad (cc',cb'/b)$$
est dans $K^{*p}\times K^{*p}$, d'o\`u le r\'esultat.
\end{proof}


\begin{thebibliography}{}











 \bibitem{Cohen} H. Cohen, 
{\em Number Theory, Volume I: Tools and Diophantine Equations}, Graduate Texts in Mathematics {\bf 239}, Springer, 2007.


\bibitem{HK} E. Halberstadt et A. Kraus, {\em Courbes de Fermat: r\'esultats et probl\`emes}, J. reine angew. Math. {\bf{548}} (2002), 167-234.

 
\bibitem{Kraus2016} A.\ Kraus,
{\em Contre-exemples au principe de Hasse pour les courbes de Fermat},
Acta Arith. \ {\bf 174} (2016), 189-197.



\bibitem{Limura} K. Limura,
{\em On the $\ell$-rank of ideal class groups of certain number fields},
Acta Arith. \ {\bf 47} (1986), 154-166.





\bibitem{Pari} The PARI~Group, PARI/GP version {\tt{2.17.1}},  
Universit\'e  de Bordeaux I, (2024).

\bibitem{Selmer} E. S. Selmer, 
{\em The diophantine equation $ax^3+by^3+cz^3=0$},
Acta Math. {\bf 85} (1951), 203-362.

\bibitem{Serre1} J.-P. Serre,
{\em Corps locaux},
Hermann,  troisi\`eme \'edition, 1968.



\bibitem{Silverman} J.\ H. Silverman,
{\em The Arithmetic of Elliptic Curves},
Second Edition, Graduate Texts in Mathematics {\bf 106}, Springer,  2009.

\bibitem{TZ} P. Tzermias, 
{\em The Group of Automorphisms of the Fermat Curve}, J. Number Theory\ {\bf 53} (1995), 173-178.

\bibitem{Weil} A. Weil, 
{\em  Sur les courbes alg\'ebriques et les vari\'et\'es qui s'en d\'eduisent}, Hermann, 1948.

\end{thebibliography}
  \end{document}